\documentclass[10pt,reqno]{amsart}
\usepackage{amsmath,amssymb,amsthm,array,bm,hyperref,tikz,enumerate,mathrsfs,multirow,setspace,verbatim}
\usepackage[vcentermath]{youngtab}
\usetikzlibrary{matrix}
\tikzset{tab/.style={matrix of math nodes,column sep=-.4, row sep=-.4,text height=8pt,text width=8pt,align=center,inner sep=3}}

\newcommand{\BB}{\mathcal{B}}
\newcommand{\bon}{\overline{1}}
\newcommand{\btw}{\overline{2}}
\newcommand{\bth}{\overline{3}}

\newcommand{\bi}{\overline\imath}
\newcommand{\bk}{\overline{k}}
\newcommand{\brr}{\overline{r}}
\newcommand{\CB}{\mathbf{B}}
\newcommand{\cc}{{\bm c}}
\newcommand{\CC}{\field{C}}
\newcommand{\cd}{\raisebox{.25ex}{$\cdot\!\cdot\!\cdot$}}
\newcommand{\field}[1]{\mathbf{#1}}
\newcommand{\ii}{{\bf i}}

\newcommand{\nz}{\operatorname{nz}}
\newcommand{\PP}{\mathcal{P}}
\newcommand{\pr}{\mathrm{pr}}

\newcommand{\QQ}{\field{Q}}
\newcommand{\R}{\mathscr{R}}
\newcommand{\RR}{\field{R}}
\newcommand{\seg}{\operatorname{seg}}
\newcommand{\TT}{\mathcal{T}}
\newcommand{\wt}{{\rm wt}}
\newcommand{\zz}{{\bm z}}
\newcommand{\ZZ}{\field{Z}}

\theoremstyle{plain}
\newtheorem{theorem}{Theorem}[section]

\newtheorem{proposition}[theorem]{Proposition}
\newtheorem{corollary}[theorem]{Corollary}
\theoremstyle{definition}
\newtheorem{definition}[theorem]{Definition}
\newtheorem{ex}[theorem]{Example}
\newtheorem{remark}[theorem]{Remark}
\numberwithin{equation}{section}
\numberwithin{figure}{section}
\numberwithin{table}{section}

\newenvironment{acknowledgements}
{\bigskip\noindent\footnotesize\textbf{Acknowldegements.} }
{\medskip}

\title[Tableaux, canonical bases, and the Gindikin-Karpelevich formula]{Young tableaux, canonical bases, and the Gindikin-Karpelevich formula}
\author{Kyu-Hwan Lee}
\address{Department of Mathematics \\ University of Connecticut \\ Storrs, CT  06269-3009}
\email{khlee@math.uconn.edu}
\urladdr{http://www.math.uconn.edu/~khlee}
\author{Ben Salisbury}
\address{Department of Mathematics \\ The City College of New York \\ New York, NY 10031 \\ and The Institute for Computational and Experimental Research in Mathematics \\ Brown University \\ Providence, RI 02903}
\email{benjamin$\underline{\phantom{s}}$salisbury@brown.edu}
\urladdr{http://www.sci.ccny.cuny.edu/~salisbury}
\date{\today}

\begin{document}
\begin{abstract}
A combinatorial description of the crystal $\mathcal{B}(\infty)$ for finite-dimen\-sional simple Lie algebras in terms of certain Young tableaux was developed by J.\ Hong and H.\ Lee.  We establish an explicit bijection between these Young tableaux and canonical bases indexed by Lusztig's parametrization, and obtain a combinatorial rule for expressing the Gindikin-Karpelevich formula as a sum over the set of Young tableaux.
\end{abstract}
\maketitle

\setcounter{section}{-1}
\section{Introduction}

The Gindikin-Karpelevich formula is a $p$-adic integration formula proved by Langlands in \cite{Lan:71}.  He named it the {\it Gindikin-Karpelevich formula} after a similar formula originally stated by Gindikin and Karpelevich \cite{GK:62} in the case of real reductive groups. The formula also appears in Macdonald's work \cite{Mac:71} on $p$-adic groups and affine Hecke algebras.

Let $G$ be a split semisimple algebraic group over a $p$-adic field $F$  with ring of integers $\mathfrak{o}_F$, and suppose the residue field $\mathfrak{o}_F/\pi\mathfrak{o}_F$ of $F$ has size $t$, where $\pi$ is a generator of the unique maximal ideal in $\mathfrak{o}_F$.  Choose a maximal torus $T$ of $G$ contained in a Borel subgroup $B$ with unipotent radical $N$, and let $N_-$ be the opposite group to $N$. We have $B=TN$.  The group $G(F)$ has a decomposition $G(F) = B(F)K$, where $K = G(\mathfrak{o}_F)$ is a maximal compact subgroup of $G(F)$.  Fix an unramified character $\tau\colon T(F) \longrightarrow \CC^\times$, and define a function $f^\circ \colon G(F) \longrightarrow \CC$ by 
\[
f^\circ (bk) = (\delta^{1/2}\tau)(b), \ \ \ b\in B(F), \ k\in K,
\]
where $\delta\colon B(F) \longrightarrow \RR^\times_{>0}$ is the modular character of $B$ and $\tau$ is extended to $B(F)$ to be trivial on $N(F)$.   The function $f^\circ$ is called the {\it standard spherical vector} corresponding to $\tau$.  

Let $G^\vee$ be the Langlands dual of $G$ with the dual torus $T^\vee$.  The set of coroots of $G$ is identified with the set of roots of $G^\vee$ and will be denoted by $\Phi$. Finally, let $\zz$ be the element of the dual torus $T^\vee$, corresponding to $\tau$ via the Satake isomorphism.

\begin{theorem}[Gindikin-Karpelevich formula, \cite{Lan:71}]
Given the setting above, we have
\begin{equation}\label{eq:GK}
\int_{N_-(F)} f^\circ(n) \,{\rm d}n =
\prod_{\alpha\in\Phi^+} \frac{1-t^{-1}\zz^\alpha}{1-\zz^\alpha},
\end{equation}
where ${\Phi}^+$ is the set of positive roots of $G^\vee$.
\end{theorem}  

Let $\mathfrak g$ be the Lie algebra of $G^\vee$, and let $\BB(\infty)$ be the crystal basis of the negative part $U_q^-(\mathfrak g)$ of the quantum group $U_q(\mathfrak g)$. Then $\Phi$ is the root system of $\mathfrak g$ as well. In recent work, the integral in the Gindikin-Karpelevich formula has been evaluated using Kashiwara's crystal basis or Lusztig's canonical basis.  
 D.\ Bump and M.\ Nakasuji \cite{BN:10} used {\it decorated string parameterizations} in the crystal $\BB(\infty)$, which are essentially paths to the highest weight vector, while in \cite{McN:11}, P.\ McNamara used a cellular decomposition of $N_-$ in bijection with Lusztig's canonical basis $\CB$ of $U_q^-(\mathfrak g)$ and computed the integral over the cells.  Both of these methods are valid for type $A_r$.

In the general case, H.\ Kim and K.-H.\ Lee \cite{KL:11} used {\it Lusztig's parameterization} of elements in $\CB$ and proved, for all finite-dimensional simple Lie algebras $\mathfrak g$, 
\begin{equation} \label{eqn-aa}
\prod_{\alpha\in\Phi^+} \frac{1-t^{-1}\zz^\alpha}{1-\zz^\alpha} = \sum_{b\in \CB} (1-t^{-1})^{\nz(\phi_\ii(b))}\zz^{-\wt(b)},
\end{equation}
where $\nz(\phi_\ii(b))$ is the number of nonzero entries in the Lusztig parametrization $\phi_\ii(b)$ of $b$ with respect to a reduced expression $\ii$ of the longest Weyl group element.

The purpose of this paper is to describe the sum in \eqref{eqn-aa} in a combinatorial way using Young tableaux.  
Since the canonical basis $\CB$ is the same as Kashiwara's global crystal basis, we may replace $\CB$ with $\BB(\infty)$.  The associated crystal structure on $\CB$ may be described in terms of the Lusztig parametrization \cite{BZ:01,Luszt:93} or the string parametrization \cite{BZ:96,Kash:93,Lit:98}, and there are formulas relating the two given by Berenstein and Zelevinsky in \cite{BZ:01}. 
Much work has been done on realizations of crystals (e.g., \cite{Kam:10,Kang:03,KN:94,KS:97,Lit:95}). In the case of $\BB(\infty)$ for finite-dimensional simple Lie algebras, J.\ Hong and H.\ Lee used {\it marginally large} semistandard Young tableaux to obtain a realization of crystals \cite{HL:08}.  We will use their marginally large semistandard Young tableaux realization of $\BB(\infty)$ to write the right-hand side of \eqref{eqn-aa} as a sum over a set $\TT(\infty)$ of tableaux.  It turns out that the appropriate data to define the coefficient comes from a consecutive string of letters $k$ in the tableaux, which we call a {\em $k$-segment}. Define $\seg(T)$ to be the total number of $k$-segments in a tableau $T$ for types $A_r$ and $C_r$. For other types, see  Definition \ref{def-seg} (2). Our result is the following.

\begin{theorem}\label{thm:main}
Let $\mathfrak{g}$ be a Lie algebra of type $A_r$, $B_r$, $C_r$, $D_r$, or $G_2$.  Then
\begin{equation} \label{eqn-first}
\prod_{\alpha\in\Phi^+} \frac{1-t^{-1}\zz^\alpha}{1-\zz^\alpha} = \sum_{T\in \TT(\infty) } (1-t^{-1})^{\seg(T)}\zz^{-\wt(T)}.
\end{equation}
\end{theorem}

The point is that the exponent $\seg(T)$ can be read off immediately from the tableau $T$.  In \cite{LS-A}, the authors achieved this result when $\mathfrak{g}$ is of type $A_r$, where the method of proof first recovers the string parametrization of a tableau from the lengths of $k$-segments. In this paper, we will adopt a different approach.  We construct a bijection from $\TT(\infty)$ to the set of Kostant partitions and use the natural bijection from the set of Kostant partitions to Lusztig's canonical basis $\CB$ (see the diagram in \eqref{diagram}).  In this way, we relate a $k$-segment of a tableau $T$ with a particular positive root up to some necessary modifications.  This idea is similar to the approach used by the authors together with S.-J.\ Kang and H.\ Ryu in the type $A_r^{(1)}$ case \cite{KLRS-A}.

There is a companion formula to the Gindikin-Karpelevich formula, called the Cassel\-man-Shalika formula, which may be viewed as the highest weight crystal analogue of our work here.  The corresponding type $A_r$ result to this work for the Casselman-Shalika may be found in \cite{LLS-A}.  It is also worth noting that there are well-known bijections between the Lusztig parametrization, string parametrization, and semistandard Young tableaux in type $A_r$.  More details may be found in \cite{MG:08,Sav:06}.
 
The outline of this paper is as follows.  In Section \ref{sec:crystal}, we set our basic notation and review the notion of a combinatorial crystal and its properties.  In Section \ref{sec:binf}, we recall the description of $\BB(\infty)$ crystal given by marginally large semistandard Young tableaux according to J.\ Hong and H.\ Lee.  The definition of $\seg(T)$ and the proof of Theorem \ref{thm:main} will be presented in Section \ref{sec:main}.   Section \ref{sec:app} gives some applications to the study of symmetric functions.

\begin{acknowledgements}
B.\;S.\ would like to thank Gautam Chinta for his support during a portion of this work.  He would also like to thank Travis Scrimshaw for his help during the development of the $\TT(\infty)$ implementation in Sage.  This latter development was completed while both authors were visiting ICERM during the Spring 2013 semester program entitled ``Automorphic Forms, Combinatorial Representation Theory, and Multiple Dirichlet Series.''
\end{acknowledgements}

\section{General definitions}\label{sec:crystal}

Let $I$ be a finite index set and let $\mathfrak{g}$ be a finite-dimensional simple complex Lie algebra of rank $r:= \#I \ge 1$ with simple roots $\{ \alpha_i : i\in I\}$ and Cartan matrix $A = (a_{ij})_{i,j\in I}$.  We denote the generators of $\mathfrak{g}$ by $e_i$, $f_i$, and $h_i$, for $i\in I$.  Let $P = \bigoplus_{i\in I}\ZZ\omega_i$ and $P^+ = \bigoplus_{i\in I}\ZZ_{\ge0}\omega_i$ be the weight lattice and dominant integral weight lattice, respectively, where $\omega_i$ $(i\in I) $ are the fundamental weights of $\mathfrak{g}$.    Let $\{h_i : i\in I\}$ denote the set of coroots of $\mathfrak{g}$, and recall the pairing $\langle \ ,\ \rangle \colon P^\vee \times P \longrightarrow \ZZ$ by $\langle h,\lambda \rangle = \lambda(h)$ with the condition that $a_{ij} = \alpha_j(h_i)$, where $P^\vee = \bigoplus_{i\in I} \ZZ h_i$ is the dual weight lattice.  The Cartan subalgebra of $\mathfrak{g}$ is $\mathfrak{h} = \CC\otimes_\ZZ P^\vee$, and its dual is $\mathfrak{h}^* = \bigoplus_{i\in I} \CC\omega_i$.  We will denote the root lattice of $\mathfrak{g}$ by $Q = \bigoplus_{i\in I} \ZZ\alpha_i$, and the positive and negative root lattices, respectively, are $Q^+ = \bigoplus_{i\in I} \ZZ_{\ge 0}\alpha_i$ and $Q^- = -Q^+$.

Denote by $\Phi$ and $\Phi^+$, respectively, the set of roots and the set of positive roots, and define the Weyl vector $\rho$ by $2\rho = \sum_{\alpha\in\Phi^+} \alpha$.  The Weyl group of $\Phi$ is the subgroup $W\subset \operatorname{GL}(\mathfrak{h}^*)$ generated by simple reflections $\{s_i : i\in I\}$.  For each $w\in W$, there is a reduced expression $w = s_{i_1}\cdots s_{i_m}$, to which we may associate a {\it reduced word} $(i_1,\dots,i_m)$.  Let $R(w)$ denote the set of all such reduced words for a fixed $w\in W$.  In particular, we consider reduced words $\ii = (i_1,\dots,i_N) \in R(w_\circ)$, where $w_\circ$ is the longest element of $W$ and $N = \ell(w_\circ) = \#\Phi^+$.  Elements of $R(w_\circ)$ are called {\it long words}.

Let $q$ be an indeterminate, and let $U_q(\mathfrak{g})$ be the quantum group associated to $\mathfrak{g}$. An (abstract) {\it $U_q(\mathfrak{g})$-crystal} is a set $\BB$ together with maps 
\[
\wt\colon \BB \longrightarrow P, \ \ \ 
\widetilde e_i, \widetilde f_i\colon \BB \longrightarrow \BB\sqcup\{0\}, \ \ \ 
\varepsilon_i,\varphi_i\colon \BB \longrightarrow \ZZ\sqcup\{-\infty\},
\]
that satisfy a certain set of axioms (see, e.g., \cite{HK:02,Kash:95}). Of particular interest to us is the crystal $\BB(\infty)$
which is a combinatorial model of $U_q^-(\mathfrak{g})$. The crystal $\BB(\infty)$ was originally defined by Kashiwara in \cite{Kash:91}.

For the nonexceptional finite-dimensional Lie algebras, the semistandard Young tab\-leaux realization of $U_q(\mathfrak{g})$-crystals of highest weight representations $\BB(\lambda)$ with $\lambda$ a dominant integral weight, was constructed by M.\ Kashiwara and T.\ Naka\-shima \cite{KN:94}.  The $G_2$ description is due to S.-J.\ Kang and K.\ Misra \cite{KM:94}. The Young tableaux description of $\BB(\infty)$ is closely related to that of $\BB(\lambda)$ in the sense that the basic building blocks in both characterizations come from $\BB(\omega_1)$ for the fundamental weight $\omega_1$. 
The crystal graph of $\BB(\omega_1)$ is given in Figure \ref{fig:fundcrystal}.  
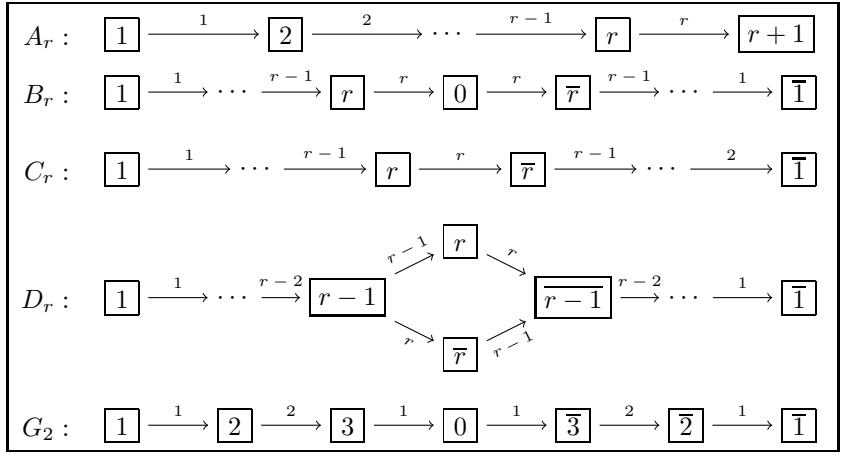
\begin{figure}[t]
\[
\begin{array}{|rl|}\hline
A_{r}: & 
\begin{tikzpicture}[scale=1.45,baseline=-4]
 \node (1) at (0,0) {$\young(1)$};
 \node (2) at (1.5,0) {$\young(2)$};
 \node (d) at (3.0,0) {$\cdots$};
 \node (n-1) at (4.5,0) {$\young(r)$};
 \node (n) at (6,0) {$\boxed{r+1}$};
 \draw[->] (1) to node[above]{\tiny$1$} (2);
 \draw[->] (2) to node[above]{\tiny$2$} (d);
 \draw[->] (d) to node[above]{\tiny$r-1$} (n-1);
 \draw[->] (n-1) to node[above]{\tiny$r$} (n);
\end{tikzpicture}\\
B_r : &
\begin{tikzpicture}[baseline=-4]
 \node (1) at (0,0) {$\young(1)$};
 \node (d1) at (1.5,0) {$\cdots$};
 \node (n) at (3,0) {$\young(r)$};
 \node (0) at (4.5,0) {$\young(0)$};
 \node (bn) at (6,0) {$\young(\brr)$};
 \node (d2) at (7.5,0) {$\cdots$};
 \node (b1) at (9,0) {$\young(\bon)$};
 \draw[->] (1) to node[above]{\tiny$1$} (d1);
 \draw[->] (d1) to node[above]{\tiny$r-1$} (n);
 \draw[->] (n) to node[above]{\tiny$r$} (0);
 \draw[->] (0) to node[above]{\tiny$r$} (bn);
 \draw[->] (bn) to node[above]{\tiny$r-1$} (d2);
 \draw[->] (d2) to node[above]{\tiny$1$} (b1);
\end{tikzpicture} 
\\[10pt]
C_r: &
\begin{tikzpicture}[baseline=-4]
 \node (1) at (0,0) {$\young(1)$};
 \node (d1) at (1.8,0) {$\cdots$};
 \node (n) at (3.6,0) {$\young(r)$};
 \node (bn) at (5.4,0) {$\young(\brr)$};
 \node (d2) at (7.2,0) {$\cdots$};
 \node (b1) at (9,0) {$\young(\bon)$};
 \draw[->] (1) to node[above]{\tiny$1$} (d1);
 \draw[->] (d1) to node[above]{\tiny$r-1$} (n);
 \draw[->] (n) to node[above]{\tiny$r$} (bn);
 \draw[->] (bn) to node[above]{\tiny$r-1$} (d2);
 \draw[->] (d2) to node[above]{\tiny$2$} (b1);
\end{tikzpicture}
\\[10pt]
D_r: & 
\begin{tikzpicture}[baseline=-4]
 \node (1) at (0,0) {$\young(1)$};
 \node (d1) at (1.5,0) {$\cdots$};
 \node (n-1) at (3,0) {$\boxed{r-1}$};
 \node (n) at (4.5,.75) {$\young(r)$};
 \node (bn) at (4.5,-.75) {$\young(\brr)$};
 \node (bn-1) at (6,0) {$\boxed{\overline{r-1}}$};
 \node (d2) at (7.5,0) {$\cdots$};
 \node (b1) at (9,0) {$\young(\bon)$};
 \draw[->] (1) to node[above]{\tiny$1$} (d1);
 \draw[->] (d1) to node[above]{\tiny$r-2$} (n-1);
 \draw[->] (n-1) to node[above,sloped]{\tiny$r-1$} (n);
 \draw[->] (n-1) to node[below,sloped]{\tiny$r$} (bn);
 \draw[->] (n) to node[above,sloped]{\tiny$r$} (bn-1);
 \draw[->] (bn) to node[below,sloped]{\tiny$r-1$} (bn-1);
 \draw[->] (bn-1) to node[above]{\tiny$r-2$} (d2);
 \draw[->] (d2) to node[above]{\tiny$1$} (b1);
\end{tikzpicture}
\\[30pt]
G_2: & 
\begin{tikzpicture}[baseline=-4]
 \node (1) at (0,0) {$\young(1)$};
 \node (2) at (1.5,0) {$\young(2)$};
 \node (3) at (3,0) {$\young(3)$};
 \node (0) at (4.5,0) {$\young(0)$};
 \node (b3) at (6,0) {$\young(\bth)$};
 \node (b2) at (7.5,0) {$\young(\btw)$};
 \node (b1) at (9,0) {$\young(\bon)$};
 \path[->,font=\tiny]
  (1) edge node[above]{$1$} (2)
  (2) edge node[above]{$2$} (3)
  (3) edge node[above]{$1$} (0)
  (0) edge node[above]{$1$} (b3)
  (b3) edge node[above]{$2$} (b2)
  (b2) edge node[above]{$1$} (b1);
\end{tikzpicture}\\\hline
\end{array}
\]
\caption{The fundamental crystals $\BB(\omega_1)$ when the underlying Lie algebra is of finite type.}\label{fig:fundcrystal}
\end{figure}

\section{A combinatorial realization of $\BB(\infty)$}\label{sec:binf}

This section is a summary of the results from \cite{HL:08}.  Recall that a tableaux $T$ is {\it semistandard} (with respect to an alphabet $J$; i.e., a totally ordered set) if entries are weakly increasing in rows from left to right and strictly increasing in columns from top to bottom.  J.\ Hong and H.\ Lee define a tableau $T$ to be {\it marginally large} if, for all $1 \le i \le r$, the number of $i$-boxes in the $i$th row of $T$ is greater than the number of all boxes in the $(i+1)$st row by exactly one.  Following \cite{HL:08}, we present the set $\TT(\infty)$ type-by-type. 

\subsection{Type $A$} When $\mathfrak{g}$ is of type $A_r$, $\TT(\infty)$ is the set of marginally large semistandard tableaux on the alphabet
\[
J(A_r) := \{ 1 \prec 2 \prec \cdots \prec r \prec r+1 \}
\]
satisfying the following conditions.
\begin{enumerate}
\item Each tableaux has exactly $r$ rows.
\item The first column has entries $1, 2,  \dots, r$.
\end{enumerate}

\begin{ex}
For $\mathfrak{g}$ of type $A_3$, the elements of $\TT(\infty)$ all have the form
\[
T = \begin{tikzpicture}[baseline]
\matrix [matrix of math nodes,column sep=-.4, row sep=-.4,text height=8, inner sep=3] 
 {
 	\node[draw,fill=gray!30]{1}; & 
	\node[draw,fill=gray!30]{1 \cdots 1}; & 
	\node[draw,fill=gray!30]{1}; & 
	\node[draw,fill=gray!30]{1\cdots 1}; & 
	\node[draw,fill=gray!30]{1\cdots 1}; & 
	\node[draw,fill=gray!30]{1}; &
	\node[draw]{2\cdots 2}; & 
	\node[draw]{3\cdots 3}; & 
	\node[draw]{4\cdots 4}; \\
	\node[draw,fill=gray!30]{2}; & 
	\node[draw,fill=gray!30]{2\cdots 2}; &
	\node[draw,fill=gray!30]{2}; & 
	\node[draw]{3 \cdots 3}; & 
	\node[draw]{4 \cdots 4}; \\
  	\node[draw,fill=gray!30]{3}; & 
	\node[draw]{4 \cdots 4}; \\
 };
\end{tikzpicture},
\]
where the shaded parts are the required parts and the unshaded parts are variable.  In particular, the unique element of weight zero in this crystal is
\[
T_\infty = \begin{tikzpicture}[baseline]
\matrix [tab] 
 {
 	\node[draw,fill=gray!30]{1}; & 
	\node[draw,fill=gray!30]{1}; & 
	\node[draw,fill=gray!30]{1}; \\
  	\node[draw,fill=gray!30]{2}; & 
	\node[draw,fill=gray!30]{2};  \\
  	\node[draw,fill=gray!30]{3};  \\
 };
\end{tikzpicture}.
\]
\end{ex}

\subsection{Type $B$} When $\mathfrak{g}$ is of type $B_r$, $\TT(\infty)$ is the set of marginally large semistandard tableaux on the alphabet
\[
J(B_r) := \{ 1 \prec \cdots \prec r \prec 0 \prec \overline r \prec \cdots \prec \overline 1\}
\]
satisfying the following conditions.
\begin{enumerate}
\item Each tableaux has exactly $r$ rows.
\item The first column has entries $1, 2,  \dots, r$.
\item Contents of each box in the $i$th row is less than or equal to $\overline \imath$ (with respect to $\prec$).
\item A $0$-box occurs at most once in each row.
\end{enumerate}

\begin{ex}
For $\mathfrak{g}$ of type $B_3$, the elements of $\TT(\infty)$ all have the form
\[
T = \begin{tikzpicture}[baseline,font=\small]
\matrix [matrix of math nodes,column sep=-.4, row sep=-.4,text height=8, inner sep=2.5] 
 {
 	\node[draw,fill=gray!30]{1}; & 
	\node[draw,fill=gray!30]{1}; & 
	\node[draw,fill=gray!30]{1 \cdots 1}; & 
	\node[draw,fill=gray!30]{1}; & 
	\node[draw,fill=gray!30]{1\cdots 1}; & 
	\node[draw,fill=gray!30]{1}; & 
	\node[draw,fill=gray!30]{1\cdots1}; & 
	\node[draw,fill=gray!30]{1\cdots1}; & 
	\node[draw,fill=gray!30]{1}; & 
	\node[draw]{2\cdots 2}; & 
	\node[draw]{3\cdots 3}; & 
	\node[draw]{0}; &
	\node[draw]{\overline 3 \cdots \overline 3}; & 
	\node[draw]{\overline 2\cdots \overline 2}; & 
	\node[draw]{\overline 1 \cdots \overline 1};\\
  	\node[draw,fill=gray!30]{2}; & 
	\node[draw,fill=gray!30]{2}; & 
	\node[draw,fill=gray!30]{2\cdots 2}; & 
	\node[draw,fill=gray!30]{2}; & 
	\node[draw]{3 \cdots 3}; & 
	\node[draw]{0}; & 
	\node[draw]{\overline 3\cdots \overline 3}; &
	\node[draw]{\overline 2\cdots \overline 2}; \\
  	\node[draw,fill=gray!30]{3}; & 
	\node[draw]{0}; &  
	\node[draw]{\overline 3 \cdots \overline 3}; \\
 };
\end{tikzpicture},
\]
where the shaded parts are the required parts and the unshaded parts are variable.  In particular, the unique element of weight zero in this crystal is
\[
T_\infty = \begin{tikzpicture}[baseline]
\matrix [tab] 
 {
 	\node[draw,fill=gray!30]{1}; & 
	\node[draw,fill=gray!30]{1}; & 
	\node[draw,fill=gray!30]{1}; \\
  	\node[draw,fill=gray!30]{2}; & 
	\node[draw,fill=gray!30]{2};  \\
  	\node[draw,fill=gray!30]{3};  \\
 };
\end{tikzpicture}.
\]
\end{ex}

\subsection{Type $C$} In type $C_r$, $\TT(\infty)$ is the set of marginally large semistandard tableaux on the alphabet
\[
J(C_r) := \{ 1 \prec \cdots \prec r \prec \overline r \prec \cdots \prec \overline 1\}
\]
satisfying the following conditions.
\begin{enumerate}
\item Each tableaux has exactly $r$ rows.
\item The first column has entries $1, 2,  \dots, r$.
\item Contents of each box in the $i$th row is less than or equal to $\overline\imath$ (with respect to $\prec$).
\end{enumerate}

\begin{ex}
For $\mathfrak{g}$ of type $C_3$, the elements of $\TT(\infty)$ all have the form
\[
T = \begin{tikzpicture}[baseline]
\matrix [matrix of math nodes,column sep=-.4, row sep=-.4,text height=8, inner sep=3] 
 {
 	\node[draw,fill=gray!30]{1}; & 
	\node[draw,fill=gray!30]{1 \cdots 1}; & 
	\node[draw,fill=gray!30]{1}; & 
	\node[draw,fill=gray!30]{1\cdots 1}; & 
	\node[draw,fill=gray!30]{1\cdots1}; & 
	\node[draw,fill=gray!30]{1\cdots1}; & 
	\node[draw,fill=gray!30]{1}; & 
	\node[draw]{2\cdots 2}; & 
	\node[draw]{3\cdots 3}; & 
	\node[draw]{\overline 3 \cdots \overline 3}; & 
	\node[draw]{\overline 2\cdots \overline 2}; & 
	\node[draw]{\overline 1 \cdots \overline 1};\\
  	\node[draw,fill=gray!30]{2}; & 
	\node[draw,fill=gray!30]{2\cdots 2}; & 
	\node[draw,fill=gray!30]{2}; & 
	\node[draw]{3 \cdots 3}; &
	\node[draw]{\overline 3\cdots \overline 3}; &
	\node[draw]{\overline 2\cdots \overline 2}; \\
  	\node[draw,fill=gray!30]{3}; & 
	\node[draw]{\overline 3 \cdots \overline 3}; \\
 };
\end{tikzpicture},
\]
where the shaded parts are the required parts and the unshaded parts are variable.  In particular, the unique element of weight zero in this crystal is
\[
T_\infty = \begin{tikzpicture}[baseline]
\matrix [tab] 
 {
 	\node[draw,fill=gray!30]{1}; & 
	\node[draw,fill=gray!30]{1}; &
 	\node[draw,fill=gray!30]{1}; \\
  	\node[draw,fill=gray!30]{2}; &
	\node[draw,fill=gray!30]{2};  \\
  	\node[draw,fill=gray!30]{3};  \\
 };
\end{tikzpicture}.
\]
\end{ex}

\subsection{Type $D$} In type $D_r$, $\TT(\infty)$ is the set of marginally large semistandard tableaux on the alphabet
\[
J(D_r) := \left\{ 1 \prec \cdots \prec r-1 \prec \begin{array}{c} r \\ \overline r \end{array} \prec \overline{r-1} \prec \cdots \prec \overline 1\right\}.
\]
satisfying the following conditions.
\begin{enumerate}
\item Each tableaux has exactly $r-1$ rows.
\item The first column has entries $1, 2,  \dots, r-1$.
\item Contents of each box in the $i$th row is less than or equal to $\overline\imath$ (with respect to $\prec$).
\item The entries $r$ and $\overline r$ do not appear in the same row.
\end{enumerate}

\begin{ex}
In type $D_4$, the elements of $\TT(\infty)$ all have the form
\[
T=\begin{tikzpicture}[baseline,font=\tiny]
\matrix [matrix of math nodes,column sep=-.4, row sep=-.5,text height=6,align=center,inner sep=1.5] 
 {
 	\node[draw,fill=gray!30]{$1$}; & 
	\node[draw,fill=gray!30]{\hspace{.75ex} $1\cdots 1$\hspace{.75ex} }; & 
	\node[draw,fill=gray!30]{$1 \cdots 1$}; & 
	\node[draw,fill=gray!30]{$1$}; & 
	\node[draw,fill=gray!30]{$1\cdots1$}; & 
	\node[draw,fill=gray!30]{\hspace{.75ex} $1\cdots1$\hspace{.75ex} }; & 
	\node[draw,fill=gray!30]{$1\cdots1$}; & 
	\node[draw,fill=gray!30]{$1\cdots1$}; & 
	\node[draw,fill=gray!30]{$1$}; & 
	\node[draw]{$2\cdots 2$}; & 
	\node[draw]{$3\cdots 3$}; & 
	\node[draw]{\raisebox{.45ex}{$x_1\cdots x_1$}}; & 
	\node[draw]{$\overline 3 \cdots \overline 3$}; & 
	\node[draw]{$\overline 2\cdots \overline 2$}; & 
	\node[draw]{$\overline 1 \cdots \overline 1$};\\
  	\node[draw,fill=gray!30]{$2$}; & 
	\node[draw,fill=gray!30]{\hspace{.75ex} $2\cdots 2$\hspace{.75ex} }; & 
	\node[draw,fill=gray!30]{$2\cdots 2$}; & 
	\node[draw,fill=gray!30]{$2$}; & 
	\node[draw]{$3 \cdots 3$}; & 
	\node[draw]{\raisebox{.45ex}{$x_2\cdots x_2$}}; & 
	\node[draw]{$\overline 3\cdots \overline 3$}; &
	\node[draw]{$\overline 2\cdots \overline 2$}; \\
  	\node[draw,fill=gray!30]{$3$}; & 
	\node[draw]{\raisebox{.45ex}{$x_3\cdots x_3$}}; & 
	\node[draw]{$\overline 3 \cdots \overline 3$}; \\
 };
\end{tikzpicture},
\]
where $x_i \in \{4,\overline4\}$ for each $i=1,2,3$, the shaded parts are the required parts, and the unshaded parts are variable.  In particular, the unique element of weight zero in this crystal is
\[
T_\infty = \begin{tikzpicture}[baseline]
\matrix [tab] 
 {
 	\node[draw,fill=gray!30]{1}; & 
	\node[draw,fill=gray!30]{1}; & 
	\node[draw,fill=gray!30]{1}; \\
  	\node[draw,fill=gray!30]{2}; & 
	\node[draw,fill=gray!30]{2}; \\
  	\node[draw,fill=gray!30]{3}; \\
 };
\end{tikzpicture}.
\]
\end{ex}

\subsection{Type $G$} Lastly, when $\mathfrak{g}$ is of type $G_2$, the elements of $\TT(\infty)$ all have the form
\[
T = \begin{tikzpicture}[baseline]
\matrix [matrix of math nodes,column sep=-.4, row sep=-.4,text height=8,inner sep=3] 
 {
  \node[draw,fill=gray!30]{1}; & 
  \node[draw,fill=gray!30]{1\cdots 1}; & 
  \node[draw,fill=gray!30]{1}; & 
  \node[draw]{2\cdots 2}; & 
  \node[draw]{3\cdots 3}; & 
  \node[draw]{0}; & 
  \node[draw]{\overline 3 \cdots \overline 3}; & 
  \node[draw]{\overline 2\cdots \overline 2}; & 
  \node[draw]{\overline 1 \cdots \overline 1};\\
  \node[draw,fill=gray!30]{2}; & 
  \node[draw]{3 \cdots 3}; \\
 };
\end{tikzpicture},
\]
where the shaded parts are the required parts and the unshaded parts are variable.  In particular, the unique element of weight zero in this crystal is
\[
T_\infty = \begin{tikzpicture}[baseline]
\matrix [tab] 
 {
 	\node[draw,fill=gray!30]{1}; & 
	\node[draw,fill=gray!30]{1}; \\
  	\node[draw,fill=gray!30]{2}; \\
 };
\end{tikzpicture}.
\]

\bigskip

A crystal structure can be defined on $\TT(\infty)$ as in \cite{HL:08} by embedding a tableau $T$ of $\TT(\infty)$ into a tensor power of the fundamental crystal $\BB(\omega_1)$ via the far-Eastern reading (where the tensor product is defined as usual \cite{HK:02,Kash:95}).  The following theorem is established by J.\ Hong and H.\ Lee.

\begin{theorem}[{\rm\cite{HL:08}}]
For underlying Lie types $A_r$, $B_r$, $C_r$, $D_r$, and $G_2$, there is a crystal isomorphism between $\TT(\infty)$ and $\BB(\infty)$.
\end{theorem}

Following D.\ Bump and M.\ Nakasuji in \cite{BN:10}, we wish to suppress the required columns from the tableaux and only include the variable parts.   This convention will save space, making drawing the graphs easier and it will help make the $k$-segments, to be defined later, stand out.  We will call this modification of $T\in \TT(\infty)$ the {\it reduced form} of $T$, and denote it by $T^\sharp$, for $T\in \TT(\infty)$.  For example, in type $C_3$, we have
\[
\left(\, 
\begin{tikzpicture}[baseline]
\matrix [tab] 
 {
 	\node[draw,fill=gray!30]{1}; & 
	\node[draw,fill=gray!30]{1}; & 
	\node[draw,fill=gray!30]{1}; \\
  	\node[draw,fill=gray!30]{2}; & 
	\node[draw,fill=gray!30]{2};  \\
  	\node[draw,fill=gray!30]{3};  \\
 };
\end{tikzpicture}
\,\right)^\sharp =
\begin{tikzpicture}[baseline]
\matrix [tab] 
 {
 	\node[draw]{*}; \\
  	\node[draw]{*}; \\
  	\node[draw]{*}; \\
 };
\end{tikzpicture}
\]
and, in $C_2$, 
\[
\left(\,
\begin{tikzpicture}[baseline]
\matrix [tab] 
 {
 	\node[draw,fill=gray!30]{1}; & 
	\node[draw,fill=gray!30]{1}; & 
	\node[draw,fill=gray!30]{1}; &
	\node[draw,fill=gray!30]{1}; & 
	\node[draw,fill=gray!30]{1}; & 
	\node[draw]{2}; & 
	\node[draw]{2}; \\
  	\node[draw,fill=gray!30]{2}; & 
	\node[draw]{\btw}; & 
	\node[draw]{\btw}; & 
	\node[draw]{\btw};  \\
 };
\end{tikzpicture}
\,\right)^\sharp =  
\begin{tikzpicture}[baseline]
\matrix [tab] 
 {
 	\node[draw]{2}; & 
	\node[draw]{2}; \\
  	\node[draw]{\btw}; & 
	\node[draw]{\btw}; & 
	\node[draw]{\btw};  \\
 };
\end{tikzpicture}\ ,
\]
where $*$ is used to denote a row without any variable entries.
In particular, the resulting shape need not be a Young diagram.  Note that there is no essential information lost when passing to the reduced form.  Set $\TT(\infty)^\sharp = \{ T^\sharp : T\in \TT(\infty) \}$.  
We conclude with some examples of $\TT(\infty)$ crystals, of course with only the top part of the graph computed.  
See Figures \ref{fig:B3} and \ref{fig:G2} for $\TT(\infty)^\sharp$ when $\mathfrak{g}$ is of type $B_3$ and $G_2$, respectively.

\begin{figure}[t]
\centering
\begin{tikzpicture}[scale=2.5]
\node (0) at (0,0) {$\young(*,*,*)$};
\node (1) at (-1,-1) {$\young(2,*,*)$};
\node (2) at (0,-1)  {$\young(*,3,*)$};
\node (3) at (1,-1)  {$\young(*,*,0)$};
\node (11) at (-2,-2) {$\young(22,*,*)$};
\node (12) at (-1,-2) {$\young(3,*,*)$};
\node (13) at (-.5,-2){$\young(2,*,0)$};
\node (21) at (-1.25,-2){$\young(2,3,*)$};
\node (22) at (0,-2)  {$\young(*,33,*)$};
\node (23) at (.75,-2) {$\young(*,0,*)$};
\node (32) at (1,-2)  {$\young(*,3,0)$};
\node (33) at (2,-2)  {$\young(*,*,\bth)$};
\path[->,font=\scriptsize,inner sep=1]
	(0) edge node[midway,fill=white]{$1$} (1.90)
	    edge node[midway,fill=white]{$2$} (2.90)
	    edge node[midway,fill=white]{$3$} (3.90)
    (3) edge node[near start,fill=white]{$1$} (13.70)
        edge node[midway,fill=white]{$2$} (32.90)
        edge node[midway,fill=white]{$3$} (33.120)
    (2) edge node[near start,fill=white]{$1$} (21.90)
        edge node[midway,fill=white]{$2$} (22.90)
        edge node[midway,fill=white]{$3$} (23.90)
	(1) edge node[midway,fill=white]{$1$} (11.90)
        edge node[midway,fill=white]{$2$} (12.90)
        edge node[midway,fill=white]{$3$} (13.110);
\end{tikzpicture}
\caption{The top part of $\mathcal{T}(\infty)^\sharp$ in type $B_3$.}
\label{fig:B3}
\end{figure}
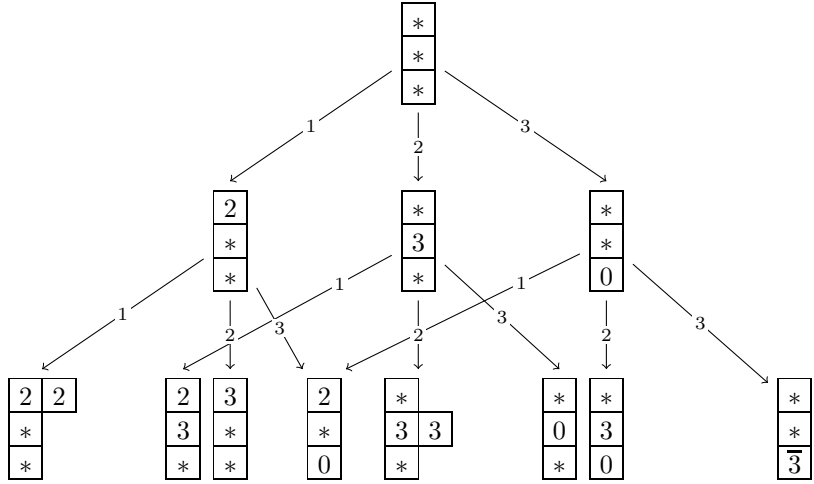

\begin{figure}[t]
\centering
\begin{tikzpicture}[xscale=1.8,yscale=2]
\node (0) at (0,0) {$\young(*,*)$};
\node (11) at (-1,-1) {$\young(2,*)$};
\node (12) at (1,-1) {$\young(*,3)$};
\node (21) at (-2,-2) {$\young(22,*)$};
\node (22) at (-.25,-2) {$\young(3,*)$};
\node (23) at (.25,-2) {$\young(2,3)$};
\node (24) at (2,-2) {$\young(*,33)$};
\node (31) at (-3,-3) {$\young(222,*)$};
\node (32) at (-1.75,-3) {$\young(23,*)$};
\node (33) at (-1.2,-3) {$\young(0,*)$};
\node (34) at (-.65,-3) {$\young(22,3)$};
\node (35) at (.7,-3) {$\young(3,3)$};
\node (36) at (1.6,-3) {$\young(2,33)$};
\node (37) at (3,-3) {$\young(*,333)$};
\path[->,font=\scriptsize,inner sep=1]
 (0) edge node[midway,fill=white]{$1$} (11)
 (0) edge node[midway,fill=white]{$2$} (12)
 (11) edge node[midway,fill=white]{$1$} (21)
 (11) edge node[midway,fill=white]{$2$} (22)
 (12) edge node[midway,fill=white]{$1$} (23)
 (12) edge node[midway,fill=white]{$2$} (24)
 (21) edge node[midway,fill=white]{$1$} (31)
 (21) edge node[midway,fill=white]{$2$} (32)
 (22) edge node[midway,fill=white]{$1$} (33.90)
 (22) edge node[midway,fill=white]{$2$} (35.90)
 (23) edge node[midway,fill=white]{$1$} (34.90)
 (23) edge node[midway,fill=white]{$2$} (36)
 (24) edge node[midway,fill=white]{$1$} (36)
 (24) edge node[midway,fill=white]{$2$} (37);
\end{tikzpicture}
\caption{The top part of $\mathcal{T}(\infty)^\sharp$ in type $G_2$.}
\label{fig:G2}
\end{figure}

\section{Main result}\label{sec:main}

In this section, $X_r =A_r, B_r$, $C_r$, $D_r$, or $G_2$.  This first definition is a generalization of that given in \cite{LS-A}.

\begin{definition} \label{def-seg}
Let $T\in \TT(\infty)$.  
\begin{enumerate}
\item Define a {\it $k$-segment}, $k \in J(X_r)\setminus\{1\}$, to be a maximal sequence of $k$-boxes in one row of $T$.
By definition, we do not consider the required collection of $k$-boxes beginning the $k$th row of $T$ to be a $k$-segment.
\item Let $\seg'(T)$ be the total number of segments of $T$. We define $\seg(T)$ type-by-type.

\begin{itemize}
\item In type $A_r$ or $C_r$, we simply define $\seg(T) = \seg'(T)$.

\item In type $B_r$, we set $e_B(T)$ to be the number of rows $i$ that contain both $0$-segment and $\overline \imath$-segment, and define $\seg(T)=\seg'(T) - e_B(T)$.

\item In type $D_r$, we set $e_D(T)$ to be the number of rows $i$ that contain $\overline \imath$-segment but neither $r$- nor $\overline r$-segment, and define
\[ \seg(T)= \seg'(T) + e_D(T) .\]

\item In type $G_2$, if the first row contains both $0$-segment and $\overline 1$-segment, we define $\seg(T) =\seg'(T)-1$; otherwise, $\seg(T) =\seg'(T)$.
\end{itemize}
\end{enumerate}
\end{definition}

\begin{ex}
Let $X_r = A_3$ and 
\[
T = \begin{tikzpicture}[baseline]
\matrix [tab] 
 {
 	\node[draw,fill=gray!30]{1}; & 
	\node[draw,fill=gray!30]{1}; & 
	\node[draw,fill=gray!30]{1}; & 
	\node[draw,fill=gray!30]{1}; & 
	\node[draw,fill=gray!30]{1}; & 
	\node[draw,fill=gray!30]{1}; & 
	\node[draw,fill=gray!30]{1}; & 
	\node[draw,fill=gray!30]{1}; & 
	\node[draw,fill=gray!30]{1}; & 
	\node[draw]{2}; & 
	\node[draw]{2}; & 
	\node[draw]{3}; &
	\node[draw]{3}; & \\
  	\node[draw,fill=gray!30]{2}; & 
	\node[draw,fill=gray!30]{2}; & 
	\node[draw,fill=gray!30]{2}; & 
	\node[draw,fill=gray!30]{2}; & 
	\node[draw]{3}; & 
	\node[draw]{3}; & 
	\node[draw]{4}; &
	\node[draw]{4}; \\
  	\node[draw,fill=gray!30]{3}; & 
	\node[draw]{4}; &  
	\node[draw]{4}; \\
 };
\end{tikzpicture}.
\]
Then $\seg(T) =\seg'(T)= 5$, since there are a $2$-segment and a $3$-segment in the first row, a $3$-segment and a $4$-segment in the second row, and a $4$-segment in the third row.
\end{ex}

\begin{ex}\label{ex:bseg}
Let $X_r = B_3$ and 
\[
T = \begin{tikzpicture}[baseline]
\matrix [tab] 
 {
 	\node[draw,fill=gray!30]{1}; & 
	\node[draw,fill=gray!30]{1}; & 
	\node[draw,fill=gray!30]{1}; & 
	\node[draw,fill=gray!30]{1}; & 
	\node[draw,fill=gray!30]{1}; & 
	\node[draw,fill=gray!30]{1}; & 
	\node[draw,fill=gray!30]{1}; & 
	\node[draw,fill=gray!30]{1}; & 
	\node[draw,fill=gray!30]{1}; & 
	\node[draw]{2}; & 
	\node[draw]{2}; & 
	\node[draw]{0}; &
	\node[draw]{\overline 3}; &
	\node[draw]{\overline 1}; & 
	\node[draw]{\overline 1}; & 
	\node[draw]{\overline 1};\\
  	\node[draw,fill=gray!30]{2}; & 
	\node[draw,fill=gray!30]{2}; & 
	\node[draw,fill=gray!30]{2}; & 
	\node[draw,fill=gray!30]{2}; & 
	\node[draw]{3}; & 
	\node[draw]{3}; & 
	\node[draw]{\overline 2}; &
	\node[draw]{\overline 2}; \\
  	\node[draw,fill=gray!30]{3}; & 
	\node[draw]{0}; &  
	\node[draw]{\overline 3}; \\
 };
\end{tikzpicture}.
\]
Immediately, we see $\seg'(T)=8$. The first row has a $0$-segment and a $\overline 1$-segment, and the third row has both a $0$-segment and a $\overline 3$-segment. Thus $e_B(T)=2$, and we obtain
$\seg(T) = \seg'(T) -e_B(T)=6$.
\end{ex}

\begin{ex}\label{ex:dseg}
Let $X_r = D_4$ and
\[
T = \begin{tikzpicture}[baseline]
\matrix [tab] 
 {
 	\node[draw,fill=gray!30]{1}; & 
	\node[draw,fill=gray!30]{1}; & 
	\node[draw,fill=gray!30]{1}; & 
	\node[draw,fill=gray!30]{1}; & 
	\node[draw,fill=gray!30]{1}; & 
	\node[draw,fill=gray!30]{1}; & 
	\node[draw,fill=gray!30]{1}; & 
	\node[draw,fill=gray!30]{1}; & 
	\node[draw,fill=gray!30]{1}; & 
	\node[draw]{2}; & 
	\node[draw]{2}; &
	\node[draw]{\overline 3}; &
	\node[draw]{\overline 1}; & 
	\node[draw]{\overline 1}; & 
	\node[draw]{\overline 1};\\
  	\node[draw,fill=gray!30]{2}; & 
	\node[draw,fill=gray!30]{2}; & 
	\node[draw,fill=gray!30]{2}; & 
	\node[draw,fill=gray!30]{2}; & 
	\node[draw]{3}; & 
	\node[draw]{\overline 4}; &
	\node[draw]{\overline 3}; &
	\node[draw]{\overline 3}; \\
  	\node[draw,fill=gray!30]{3}; & 
	\node[draw]{\overline 4}; &  
	\node[draw]{\overline 3}; \\
 };
\end{tikzpicture}.
\]
Clearly, $\seg'(T)=8$. 
The first row has a $\overline 1$-segment but no $4$- nor $\overline 4$-segment. The second row does not have a $\overline 2$-segment, while the third row has a $\overline 3$-segment and a $\overline 4$-segment.  Thus $e_D(T)=1$, where the sole contribution comes from the first row. Then we have
\[
\seg(T) = \seg'(T) + e_D(T) = 8+1=9. 
\]
\end{ex}

\begin{ex}\label{ex:gseg}
Let $X_r = G_2$ and
\[
T = \begin{tikzpicture}[baseline]
\matrix [tab] 
 {
 	\node[draw,fill=gray!30]{1}; & 
	\node[draw,fill=gray!30]{1}; & 
	\node[draw,fill=gray!30]{1}; & 
	\node[draw,fill=gray!30]{1}; & 
	\node[draw,fill=gray!30]{1}; & 
	\node[draw,fill=gray!30]{1}; & 
	\node[draw]{3}; & 
	\node[draw]{3}; & 
	\node[draw]{0}; &
	\node[draw]{\overline 3}; &
	\node[draw]{\overline 2}; & 
	\node[draw]{\overline 1}; & 
	\node[draw]{\overline 1};\\
  	\node[draw,fill=gray!30]{2}; & 
	\node[draw]{3}; & 
	\node[draw]{3}; & 
	\node[draw]{3}; & 
	\node[draw]{3}; & \\
 };
\end{tikzpicture}.
\]
Then $\seg(T) = \seg'(T)-1=6-1=5$, since the first row contains both a $0$-segment and a $\overline 1$-segment.
\end{ex}

Now we will relate tableaux in $\TT(\infty)$ to Kostant partitions. Set $R = \{ (\alpha) : \alpha\in\Phi^+\}$, where each $(\alpha)$ is considered as a formal symbol.  Define $\R$ to be the free abelian group generated by $R$, and let $\R^+$ be the set of the elements in $\R$ with coefficients from $\ZZ_{\ge0}$. An element of $\R^+$ should be considered as a Kostant partition, and will denoted by a boldface Greek letter; i.e., 
\[
\bm\alpha = \sum_{(\alpha)\in R} c_{(\alpha)}(\alpha).
\]  
If $c_{(\alpha)} \neq 0$, then we call $(\alpha)$ a {\em part} of $\bm\alpha$.

We define a map $\Xi\colon \TT(\infty) \longrightarrow \R^+$ by associating elements of $\R^+$ to segments of $T\in\TT(\infty)$.  In the following, $\ell_{i,k}(T)$ denotes the number of $k$-boxes in the $i$th row of $T$, and each segment is in the $i$th row except for the type $G_2$. We define $\Xi$ on a case-by-case basis.  
\begin{enumerate}[$\bullet$]

\item $\mathfrak{g}$ is of type $A_r$:
\begin{align*}
\young(k\cd k) &\mapsto
\ell_{i,k}(T)(\alpha_i + \alpha_{i+1}+ \cdots + \alpha_{k-1}) ,
\ 1 \le i < k \le r+1;
\end{align*}

\item  $\mathfrak{g}$ is of type $B_r$:
\begin{align*}
\young(k\cd k)   &\mapsto
\ell_{i,k}(T)(\alpha_i + \alpha_{i+1} + \cdots +\alpha_{k-1}) ,\ 1 \le i < k \le r,
\\
\young(0) 
  &\mapsto
(\alpha_{i}+ \alpha_{i+1}+ \cdots +\alpha_{r}), \ 1 \le i \le r, 
\\
\young(\bk\cd\bk)   &\mapsto
\ell_{i,\overline k}(T)(\alpha_i+ \cdots +\alpha_{k-1}+2\alpha_k+ \cdots +2\alpha_{r}),\ 1 \le i<k \le r,
\\
\young(\bi\cd\bi)   &\mapsto
2\ell_{i,\overline\imath}(T)(\alpha_i+ \cdots +\alpha_{r-1}+\alpha_r), \ 1 \le i \le r;
\end{align*}

\item $\mathfrak{g}$ is of type $C_r$:
\begin{align*}
\young(k\cd k)   &\mapsto
\ell_{i,k}(T)(\alpha_i + \alpha_{i+1}+ \cdots + \alpha_{k-1}) , \ 1 \le i < k \le r, 
\\
\young(\bk\cd\bk)   &\mapsto
\ell_{i,\overline k}(T)(\alpha_i + \cdots + \alpha_{r-1}+ \alpha_r + \alpha_{r-1}+ \cdots +\alpha_k),\ 1 \le i \le k \le r ;
\end{align*}

\item $\mathfrak{g}$ is of type $D_r$:
\begin{align*}
\young(k\cd k)   &\mapsto
\ell_{i,k}(T)(\alpha_i + \alpha_{i+1} + \cdots + \alpha_{k-1}) , \ 1\le i < k \le r-1, 
\\
\young(r\cd r)   &\mapsto
\ell_{i,r}(T)(\alpha_i + \alpha_{i+1} + \cdots + \alpha_{r-2}+\alpha_{r-1}) , \ 1 \le i \le r-1,
\\
\young(\brr\cd\brr)   &\mapsto
\ell_{i,\overline r}(T)(\alpha_i + \alpha_{i+1} + \cdots + \alpha_{r-2}+\alpha_r) , \ 1 \le i \le r-1, 
\\
\young(\bk\cd\bk)   &\mapsto
\ell_{i,\overline k}(T)(\alpha_i + \cdots + \alpha_{r-1}+ \alpha_{r}+\alpha_{r-2}+\cdots+\alpha_k) ,\ 1\le i <k \le r-1,
\\
\young(\bi\cd\bi)   &\mapsto
\ell_{i,\overline\imath}(T)\big((\alpha_i + \cdots + \alpha_{r-1})+(\alpha_i+\cdots+\alpha_{r-2}+\alpha_r)\big), \ 1 \le i \le r-1 ;
\end{align*}

\item $\mathfrak{g}$ is of type $G_2$:
\begin{align*}
\young(2\cd2) \text{ in the first row} &\mapsto
\ell_{1,2}(T)(\alpha_1),
\\
\young(3\cd3) \text{ in the first row} &\mapsto
\ell_{1,3}(T)(\alpha_1+\alpha_2) ,
\\
\young(0) 
\text{ in the first row} &\mapsto
(2\alpha_1 + \alpha_2),
\\
\young(\bth\cd\bth) \text{ in the first row} &\mapsto
\ell_{1,\overline 3}(T)(3\alpha_1+ \alpha_2),
\\
\young(\btw\cd\btw) \text{ in the first row} &\mapsto
\ell_{1,\overline 2}(T)(3\alpha_1+2\alpha_2),
\\
\young(\bon\cd\bon) \text{ in the first row} &\mapsto
2\ell_{1,\overline 1}(T)(2\alpha_1+\alpha_2) ,
\\
\young(3\cd3) \text{ in the second row} &\mapsto
\ell_{2,3}(T)(\alpha_2).
\end{align*}
\end{enumerate}
Then $\Xi(T)$ is defined to be the sum of elements in $\R^+$ corresponding to segments of $T$ as prescribed by the above rules.

\begin{ex}\label{ex:btheta}
Let $T$ be the type $B_3$ tableaux from Example \ref{ex:bseg}.  Then
\begin{align*}
\Xi(T) &= 
2(\alpha_1) + 
(\alpha_1+\alpha_2+\alpha_3) + 
(\alpha_1+\alpha_2+2\alpha_3) + \\
& \ \ \ \ \ \ \ \ \ 
3\cdot 2(\alpha_1+\alpha_2+\alpha_3) + 
2(\alpha_2) + 
2\cdot 2(\alpha_2+\alpha_3) +
(\alpha_3) +
2(\alpha_3) \\
&=2(\alpha_1)+7(\alpha_1+\alpha_2+\alpha_3)+(\alpha_1+\alpha_2+2\alpha_3)+2(\alpha_2)+4(\alpha_2+\alpha_3)+3(\alpha_3).
\end{align*}
\end{ex}

The following proposition is essential for the proof of the main theorem.

\begin{proposition}\label{prop:bijsegroot}
The map $\Xi\colon \TT(\infty)\longrightarrow \R^+$ is a bijection.  Moreover, $\seg(T)$ is equal to the number of distinct parts of $\Xi(T)$ for $T \in \TT(\infty)$.  
\end{proposition}

\begin{proof}
As before, let $\ell_{i,k}(T)$ be the length of the $k$-segment in the $i$th row of $T$. It is important to notice that the data $\{ \ell_{i,k}(T) \}_{i,k}$ completely determines $T$.  For the readers convenience, a list of positive roots for each type is provided in Table \ref{tab:root} (and Table \ref{tab:canonical}).
We will prove the proposition on a type-by-type basis, where, in each type, we construct a map $\Upsilon\colon \R^+ \longrightarrow \TT(\infty)$ which is the inverse of $\Xi$. Since the types $A_r$ and $C_r$ are simpler than the types $B_r$, $D_r$, and $G_2$, we deal with the types $A_r$ and $C_r$ first. 

\begin{table}[t]
\doublespacing
\[
\begin{array}{|c|c|}\hline
X_r & \Phi(X_r) \\\hline
\multirow{1}{*}{$A_r$} & \beta_{i,k}= \alpha_i + \cdots + \alpha_{k}, \ 1\le i\le k \le r  \\\hline
\multirow{2}{*}{$B_r$} & \beta_{i,k}= \alpha_i + \cdots + \alpha_{k}, \ 1\le i\le k \le r \\
& \gamma_{i,k} =\alpha_i + \cdots + \alpha_{k-1}+2 \alpha_{k}+ \cdots + 2 \alpha_r ,\ 1\le i < k \le r \\\hline
\multirow{2}{*}{$C_r$} & \beta_{i,k}= \alpha_i + \cdots + \alpha_{k}, \ 1\le i\le k \le r-1 \\
& \gamma_{i,k}= \alpha_i + \cdots + \alpha_{r-1}+\alpha_r + \alpha_{r-1} + \cdots + \alpha_k,\ 1\le i \le k \le r \\\hline
\multirow{3}{*}{$D_r$} & \beta_{i,k}= \alpha_i + \cdots + \alpha_{k}, \ 1\le i\le k \le r-1 \\
& \beta_{i,r}=\alpha_i + \cdots + \alpha_{r-2} + \alpha_r,\ 1\le i \le r-1\\
& \gamma_{i,k}=\alpha_i + \cdots + \alpha_{r-1}+ \alpha_{r} + \alpha_{r-2}+ \cdots + \alpha_k,\ 1\le i < k \le r-1\\\hline
G_2 & \alpha_1, \ \alpha_1+\alpha_2,\ 2\alpha_1+\alpha_2,\ 3\alpha_1+\alpha_2,\ 3\alpha_1+2\alpha_2,\ \alpha_2\\\hline
\end{array}
\]
\caption{Positive roots listed by type.}\label{tab:root}
\end{table}

\begin{table}[t]
\doublespacing
\[
\begin{array}{|c|c|}\hline
X_r & \Phi(X_r) \\\hline
\multirow{1}{*}{$A_r$} & \beta_{i,k}= \epsilon_i-\epsilon_{k+1}, \ 1\le i\le k \le r  \\\hline
\multirow{3}{*}{$B_r$} & \beta_{i,k}= \epsilon_i-\epsilon_{k+1}, \ 1\le i\le k \le r-1 \\
& \beta_{i,r} = \epsilon_i, \ 1\le i \le r\\
& \gamma_{i,k} = \epsilon_i+\epsilon_k ,\ 1\le i < k \le r \\\hline
\multirow{2}{*}{$C_r$} & \beta_{i,k}= \epsilon_i-\epsilon_{k+1}, \ 1\le i\le k \le r-1 \\
& \gamma_{i,k}= \epsilon_i+\epsilon_k,\ 1\le i \le  k \le r  \\\hline
\multirow{3}{*}{$D_r$} & \beta_{i,k}= \epsilon_i-\epsilon_{k+1}, \ 1\le i\le k \le r-1 \\
& \beta_{i,r}= \epsilon_i+\epsilon_r,\ 1\le i \le r-1\\
& \gamma_{i,k}= \epsilon_i+\epsilon_k,\ 1\le i < k \le r-1\\\hline
\multirow{2}{*}{$G_2$} & \epsilon_1-\epsilon_2, \ -\epsilon_1+\epsilon_3,\ -\epsilon_2+\epsilon_3,\\ 
& \epsilon_1-2\epsilon_2+\epsilon_3,\ -\epsilon_1-\epsilon_2+2\epsilon_3,\ -2\epsilon_1+\epsilon_2+\epsilon_3 \\\hline
\end{array}
\]
\singlespacing
\caption{The canonical realization of positive roots listed by type, following \cite{bourbaki}.}\label{tab:canonical}
\end{table}

\underline{Type $A_r$:} We see from Table \ref{tab:root} that $\Phi^+ = \{ \beta_{i,k} : 1 \le i \le k \le r \}$, so an element $\bm\alpha \in \R^+$ can be written as $\bm\alpha = \sum c_{i,k} (\beta_{i,k})$.  Define $\Upsilon \colon \R^+ \longrightarrow \TT(\infty)$ by setting $\Upsilon (\sum c_{i,k} (\beta_{i,k}))$ to be the tableaux such that \[ \ell_{i,k+1}(T) = c_{i,k} ,  \qquad  1 \le i \le k \le r ,\] where we write $T=\Upsilon (\sum c_{i,k} (\beta_{i,k}))$. Then it is straightforward to check that $\Upsilon$ and $\Xi$ are inverse to each other, and we also obtain that $\seg(T)$ is the number of $(\beta_{i,k})$'s with nonzero coefficient $c_{i,k}$, which is exactly the number of distinct parts of $\bm\alpha$.

\underline{Type $C_r$:}  From Table \ref{tab:root}, we have $\Phi^+ = \{ \beta_{i,k} : 1 \le i \le k \le r-1 \} \cup \{\gamma_{i,k} : 1 \le i \le k \le r \}$. Write an element $\bm\alpha \in \R^+$ as $\bm\alpha= \sum c_{i,k} (\beta_{i,k}) + \sum d_{i,k} (\gamma_{i,k})$, and define $\Upsilon(\bm\alpha)$ to be the tableau $T$ such that 
\[ \ell_{i, k+1}(T) = c_{i,k} \qquad \text{ and } \qquad \ell_{i, \overline k}(T) = d_{i,k} .\]
Then $\Upsilon$ is the inverse of $\Xi$, and $\seg(T)$ is the number of distinct parts in $\Xi(T)$.

\underline{Type $B_r$:}  In this case, for each $i\in I$, both the $0$-segment and  $\overline\imath$-segment in the $i$th row contribute the same part $(\beta_{i,r})$ from the definition of $\Xi$ above. Since we have $\Phi^+ = \{ \beta_{i,k} : 1 \le i \le k \le r \} \cup \{\gamma_{i,k} : 1 \le i < k \le r \}$, write an element $\bm\alpha \in \R^+$ as $\bm\alpha= \sum c_{i,k} (\beta_{i,k}) + \sum d_{i,k} (\gamma_{i,k})$. Define $\Upsilon(\bm\alpha)$ to be the tableau $T$ such that 
\[ 
\ell_{i, k+1}(T) = c_{i,k},  \,  
\ell_{i, \overline k}(T) = d_{i,k} , \, 
\ell_{i, \overline \imath} (T)= \left\lfloor\frac{c_{i,r}}{2}\right\rfloor  
\text{ and } 
\ell_{i,0}(T) = 
\begin{cases} 
0 & \text{if $c_{i,r}$ is even}, \\ 
1 & \text{otherwise}, 
\end{cases} 
\]
where $\lfloor n \rfloor$ denotes the largest integer less than or equal to $n$.

In order to see that $\Xi$ and $\Upsilon$ are inverse to each other, it is enough to consider $0$-segment and $\overline \imath$-segment in the $i$th row and the corresponding partition $c_{i,r}(\beta_{i,r})$. Assume that $T$ has only possibly  a $0$-segment and $\overline \imath$-segment in the $i$th row. Then $\Xi(T)= c_{i,r}(\beta_{i,r})$, where $c_{i,r}=2 \ell_{i, \overline \imath}(T) + \ell_{i,0}(T)$ with $\ell_{i,0}(T) = 0$ or $1$. Now we see that 
\[
\ell_{i,\overline \imath} \left (\Upsilon(\Xi(T)) \right ) = \left\lfloor\frac{c_{i,r}}{2}\right\rfloor = \left\lfloor \frac{2\ell_{i, \overline \imath}(T) + \ell_{i,0}(T)}{2} \right\rfloor  = \ell_{i,\overline \imath}(T)
\] 
and $\ell_{i,0} \left (\Upsilon(\Xi(T)) \right ) = \ell_{i, 0}(T)$. Thus we have $\Upsilon(\Xi(T))=T$.
Next we consider a partition $\bm\alpha=c_{i,r}(\beta_{i,r})$. Then we have 
\[
\Xi(\Upsilon(\bm\alpha))= \left(2 \left\lfloor\frac{c_{i,r}}{2}\right\rfloor + \overline{c_{i,r}}\right) (\beta_{i,r}) = c_{i,r}(\beta_{i,r})=\bm\alpha,
\]
where $\overline{c_{i,r}}=0$ if $c_{i,r}$ is even, or $\overline{c_{i,r}}=1$ otherwise.
Hence $\Upsilon$ is the inverse of $\Xi$. 

Furthermore, since $e_B(T)$ counts the number of rows $i$ such that both $\ell_{i, \overline \imath}(T)$ and $\ell_{i,0}(T)$ are nonzero, it is now clear that $\seg(T) =\seg'(T) -e_B(T)$ is the number of distinct parts in $\Xi(T)$. 

\underline{Type $D_r$:} In this type, we have $\Phi^+ = \{ \beta_{i,k} : 1 \le i \le k \le r-1 \} \cup \{ \beta_{i,r} : 1 \le i \le r-1 \} \cup \{\gamma_{i,k} : 1 \le i < k \le r-1 \}$. We need to pay attention to an $x_i$-segment ($x_i\in\{r,\overline r\}$) and $\overline \imath$-segment in the $i$th row, since these segments do not exactly match up with the corresponding parts $(\beta_{i,r-1})$ and $(\beta_{i,r})$. We write an element $\bm\alpha \in \R^+$ as $\bm\alpha= \sum c_{i,k} (\beta_{i,k}) + \sum d_{i,k} (\gamma_{i,k})$. Define $\Upsilon(\bm\alpha)$ to be the tableau $T$ such that 
\begin{align*} 
\ell_{i, k+1}(T) &= c_{i,k} \text{ for } 1 \le i \le k  \le r-2, &
\ell_{i, \overline k}(T) &= d_{i,k}  \text{ for } 1 \le i < k  \le r-1, \\
\ell_{i, r}(T) &= \max (0, c_{i, r-1} - c_{i,r}), &  \ell_{i,\overline r}(T) &= \max (0, c_{i, r} - c_{i,r-1}), \\
\ell_{i, \overline \imath} (T) &= \min (c_{i,r-1}, c_{i,r}).
\end{align*}

In order to see that $\Upsilon$ is the inverse of $\Xi$, it is enough to consider an $x_i$-segment ($x_i\in\{r,\overline r\}$) and $\overline \imath$-segment in the $i$th row and the corresponding partition $c_{i,r-1}(\beta_{i, r-1})+c_{i,r}(\beta_{i,r})$. 
Recall that, by definition, an $r$-segment and $\overline r$-segment cannot simultaneously appear in the same row of $T$. 
Assume that $T$ has only possibly  an $r$-segment and $\overline \imath$-segment in the $i$th row. Then $\Xi(T)= c_{i,r-1}(\beta_{i, r-1})+c_{i,r}(\beta_{i,r})$, where $c_{i,r-1}=\ell_{i,r}(T)+\ell_{i, \overline \imath}(T)$, $c_{i,r}=\ell_{i, \overline \imath}(T)$ and $c_{i,r-1} \ge c_{i,r}$. Write $T'=\Upsilon(\Xi(T))$. Then 
\begin{align*}
\ell_{i,r}(T') &= \max (0, c_{i, r-1} - c_{i,r})= c_{i,r-1}-c_{i,r}= \ell_{i,r}(T)\\
\ell_{i,\overline r}(T') &= \max (0, c_{i, r} - c_{i,r-1})= 0= \ell_{i, \overline r}(T),\\
\ell_{i, \overline \imath} (T')&=\min (c_{i,r-1}, c_{i,r})= c_{i,r}=\ell_{i, \overline \imath}(T).
\end{align*} 
Thus $T'=T=\Upsilon\bigl(\Xi(T)\bigr)$.

Next we assume that $\bm\alpha=c_{i,r-1}(\beta_{i, r-1})+c_{i,r}(\beta_{i,r})$ with $c_{i,r-1} \ge c_{i,r}$. Then $\Upsilon(\bm\alpha)$ does not have $\overline r$-segment, since $\ell_{i, \overline r}(\Upsilon(\bm\alpha))=0$. And we have 
\begin{align*}
\Xi(\Upsilon(\bm\alpha))&= \bigl(\max(0, c_{i, r-1} -c_{i,r})+ \min(c_{i,r-1}, c_{i,r})\bigr) (\beta_{i,r-1}) + \min(c_{i,r-1}, c_{i,r}) (\beta_{i,r}) \\ &= c_{i,r-1}(\beta_{i,r-1}) +c_{i,r}(\beta_{i,r}) =\bm\alpha.
\end{align*}
Thus $\Upsilon$ is the inverse of $\Xi$ in this case. The other case where $T$ has only possibly  an $\overline r$-segment and $\overline \imath$-segment in the $i$th row can be proved similarly.

Since $e_D(T)$ counts the number of rows $i$ that contain $\overline \imath$-segment but neither $r$- or $\overline r$-segment, we can see that $\seg(T) =\seg'(T)+e_D(T)$ is the number of distinct parts in $\Xi(T)$.

\underline{Type $G_2$:}  This is similar to the type $B_2$ case proved above, so we skip the details.
\end{proof}

Define a map $\pr\colon \R^+ \longrightarrow Q^-$ to be the (negative of the) canonical projection; i.e., 
\[
\pr \Big(\sum_{(\alpha)\in R} c_{(\alpha)} (\alpha)\Big) = -\sum_{\alpha\in\Phi^+} c_{(\alpha)} \alpha \in Q^-.
\]
Then define the map $\wt \colon \TT(\infty) \longrightarrow Q^-$ to be $\wt=\pr \circ \Xi$. This is the same function $\wt$ for the crystal structure on $\TT(\infty)$ defined by J.\ Hong and H.\ Lee.

It is well-known that for any $\ii=(i_1, i_2, \dots , i_N) \in R(w_\circ)$, we can write elements of $\Phi^+$ as
\begin{equation} \label{eqn-be}
\beta_1 = \alpha_{i_1}, \ \ \beta_2 = s_{i_1}(\alpha_{i_2}), \ \ \dots, \ \ \beta_N = s_{i_1}\cdots s_{i_{N-1}}(\alpha_{i_N}).
\end{equation}
To this data, Lusztig associates a PBW type basis $B_\ii$ of $U^-_q(\mathfrak g)$ consisting of the elements of the form
\[ 
f_{\ii}^\cc = f^{(c_1)}_{\beta_1} \cdots f^{(c_N)}_{\beta_N},
\]
where $\cc=(c_1, c_2, \dots , c_N) \in \ZZ_{\ge0}^N$, 
\[
f_{\beta_j}^{(c_j)} = T_{i_1,-1}''\cdots T_{i_{j-1},-1}''(f_{i_j}^{(c_j)}),
\]
and $f_i^{(c)}$ is the $c$th divided power of $f_i$.
(See Section 37.1.3 and Chapter 40 of \cite{Luszt:93} for more details, including the definition of $T_{i,-1}''$.)
The $\ZZ [q]$-span $\mathscr{L}$ of $B_\ii$ is independent of $\ii$. Let $\pi\colon \mathscr{L} \longrightarrow \mathscr{L}/ q \mathscr{L}$ be the natural projection. The image $\pi (B_\ii)$ is also independent of $\ii$; we denote it by $B$. The restriction of $\pi$ to $\mathscr{L} \cap \overline{\mathscr{L}}$ is an isomorphism of $\ZZ$-modules $\overline\pi\colon \mathscr{L} \cap \overline{\mathscr L} \longrightarrow \mathscr{L}/ q \mathscr{L}$, where $\overline{\phantom{a}}$ is the bar involution of $U_q(\mathfrak g)$ fixing the generators $e_i$ and $f_i$, for all $i\in I$, and sending $q\mapsto q^{-1}$.  Then the preimage $\CB = \overline\pi^{-1}(B)$ is a $\QQ(q)$-basis of $U^-_q(\mathfrak g)$, called the {\it canonical basis}.
For $\ii\in R(w_\circ)$, define a map $\phi_\ii \colon \CB \longrightarrow \ZZ_{\ge0}^N$ by setting $\phi_\ii(b) = \cc$, where $\cc \in \ZZ_{\ge0}^N$ is given by 
\[
b \equiv f_\ii^\cc \bmod q \mathscr L.
\]
Then $\phi_\ii$ is a bijection. Define $\wt(b) = - \sum_{j=1}^N c_j\beta_j \in Q^-$, and define $\nz(\phi_\ii(b))$ to be the number of nonzero $c_j$'s for $\phi_\ii(b)=(c_1, \dots , c_N)$ and $b \in \CB$.

Now we can state the main theorem of this paper.

\begin{theorem} \label{main}
Let $\mathfrak g$ be a Lie algebra of type $A_r, B_r, C_r, D_r$ or $G_2$ and fix a long word $\ii \in R(w_\circ)$.  Then we have 
\begin{align} 
\prod_{\alpha\in\Phi^+} \frac{1-t^{-1}\zz^\alpha}{1-\zz^\alpha} &= \sum_{b\in \CB} (1-t^{-1})^{\nz(\phi_\ii(b))}\zz^{-\wt(b)}  \label{eqn-nu}\\ 
&= \sum_{T\in \TT(\infty) } (1-t^{-1})^{\seg(T)}\zz^{-\wt(T)}. \nonumber
\end{align}
\end{theorem}

\begin{proof}
The first equality is Proposition 1.4 in \cite{KL:11}, so we need only prove the second equality.
Since a long word $\ii \in R(w_\circ)$ is fixed, the positive roots $\beta_1, \dots, \beta_N$ are determined as in \eqref{eqn-be}.
Define a map $\Psi_\ii\colon \CB \longrightarrow \R^+$ by
\[ 
\Psi_\ii(b)= \sum_{j=1}^N c_j (\beta_j),
\]
where $\phi_\ii(b)= (c_1, \dots , c_N)$.
Since the weight space decomposition of $U^-_q(\mathfrak g)$ is preserved under the classical limit $q \rightarrow 1$ (see, e.g., \cite{HK:02}), the theory of Kostant partitions for the negative part $U^-(\mathfrak g)$ of the universal enveloping algebra $U(\mathfrak g)$ tells us that $\Psi_\ii$ is a bijection. Moreover, $\nz(\phi_\ii(b))$ is the same as the number of distinct parts in $\Psi_\ii(b)$ by construction. 
So we have the following diagram.
\begin{equation}\label{diagram}
\begin{tikzpicture}[baseline]
\matrix (m) [matrix of math nodes, row sep=.4in, column sep=.6in, text height=1.5ex, text depth=0.25ex]
 {\TT(\infty) && \CB\\
  & \R^+ & \\};
\path[->,font=\scriptsize]
 (m-1-1) edge node[above]{$\Theta_\ii=\Psi_\ii^{-1}\circ\Xi$} (m-1-3)
 (m-1-1) edge node[below left]{$\Xi$} (m-2-2)
 (m-1-3) edge node[below right]{$\Psi_\ii$} (m-2-2)
  ;
\end{tikzpicture}
\end{equation}
By Proposition \ref{prop:bijsegroot}, the map $\Theta_\ii:=\Psi_\ii^{-1}\circ \Xi$ defines a bijection between $\TT(\infty)$ and $\CB$. Since $\wt=\pr \circ \Psi_\ii$ on $\CB$, we have 
\begin{equation} \label{eqn-wt}
\wt (\Theta_\ii(T))= (\pr \circ \Psi_\ii \circ \Theta_\ii) (T)= (\pr \circ \Xi) (T) = \wt (T). 
\end{equation} 
We also have 
\begin{equation} \label{eqn-same}
\seg(T) = \nz(\Theta_\ii(T))
\end{equation}
for all $T\in\TT(\infty)$, since each of $\seg(T)$ and $\nz(\Theta_\ii(T))$ is equal to the number of distinct parts of $\Xi(T)$ by Proposition \ref{prop:bijsegroot} and the observation made above.
Finally, applying the bijection $\Theta_\ii$ to \eqref{eqn-nu}, we replace $\CB$ with $\TT(\infty)$, $\nz(\phi_\ii(b))$ with $\seg(T)$, and $\wt(b)$ with $\wt(T)$ to complete the proof.
\end{proof}

We have obtained the following corollary, which is of its own interest.

\begin{corollary} \label{cor:last}
For each $\ii \in R(w_\circ)$, the map $\Theta_\ii\colon \TT(\infty)\longrightarrow \CB$ is a bijection such that 
\begin{align*}
\wt(T)  &= \wt (\Theta_\ii(T)), \\
\seg(T) &= \nz(\phi_\ii(\Theta_\ii(T))).
\end{align*}
\end{corollary}

\begin{remark}
The map $\Theta_\ii$ is not a crystal isomorphism in general.
\end{remark}

\begin{ex}
Consider $T$ from Example \ref{ex:bseg} and choose $\ii= (3,2,3,2,1,2,3,2,1)$. Then we have
\[ 
\begin{array}{c} 
\beta_1= \alpha_3,\ 
\beta_2= \alpha_2+2\alpha_3,\ 
\beta_3= \alpha_2+\alpha_3,\ 
\beta_4=\alpha_2 ,\ 
\beta_5= \alpha_1+2\alpha_2+2\alpha_3 ,\\ 
\beta_6=\alpha_1+\alpha_2+2\alpha_3 ,\ 
\beta_7= \alpha_1+\alpha_2+\alpha_3 , \ 
\beta_8= \alpha_1+\alpha_2 ,\ 
\beta_9= \alpha_1 . 
\end{array}
\]
From Example \ref{ex:btheta}, we obtain
\[ 
\phi_{\ii}(\Theta_\ii(T)) = (3, 0, 4,2, 0, 1, 7, 0,2),
\] 
and see that $\seg (T) = \nz(\phi_\ii(\Theta_\ii(T))) =6$.
\end{ex}

\begin{figure}[t]
\centering
\begin{tikzpicture}[xscale=2.5,yscale=2.5]
\node (0) at (0,0) {$1$};
\node (1) at (-1,-1) {$1-t^{-1}$};
\node (2) at (0,-1)  {$1-t^{-1}$};
\node (3) at (1,-1)  {$1-t^{-1}$};
\node (11) at (-2,-2) {$1-t^{-1}$};
\node (12) at (-1,-2.1) {$1-t^{-1}$};
\node (13) at (-.65,-1.9){$(1-t^{-1})^2$};
\node (21) at (-1.35,-1.9){$(1-t^{-1})^2$};
\node (22) at (0,-2.1)  {$1-t^{-1}$};
\node (23) at (.8,-2.1) {$1-t^{-1}$};
\node (32) at (1,-1.9)  {$(1-t^{-1})^2$};
\node (33) at (2,-2)  {$1-t^{-1}$};
\path[->,font=\scriptsize,inner sep=1]
	(0) edge node[midway,fill=white]{$1$} (1.90)
	    edge node[midway,fill=white]{$2$} (2.90)
	    edge node[midway,fill=white]{$3$} (3.90)
    (3) edge node[near start,fill=white]{$1$} (13.70)
        edge node[midway,fill=white]{$2$} (32.90)
        edge node[midway,fill=white]{$3$} (33.120)
    (2) edge node[near start,fill=white]{$1$} (21.90)
        edge node[midway,fill=white]{$2$} (22.90)
        edge node[midway,fill=white]{$3$} (23.150)
	(1) edge node[midway,fill=white]{$1$} (11.90)
        edge node[midway,fill=white]{$2$} (12.90)
        edge node[midway,fill=white]{$3$} (13.110);
\end{tikzpicture}
\caption{The coefficients for the top part of $\mathcal{T}(\infty)$ in type $B_3$. Compare with Figure \ref{fig:B3}.}
\label{fig:B3t}
\end{figure}

\begin{figure}[t]
\centering
\begin{tikzpicture}[yscale=1.75,xscale=1.8]
\node (0) at (0,0) {$1$};
\node (11) at (-1,-1) {$1-t^{-1}$};
\node (12) at (1,-1) {$1-t^{-1}$};
\node (21) at (-2,-2) {$1-t^{-1}$};
\node (22) at (-.45,-2) {$1-t^{-1}$};
\node (23) at (.45,-2) {$(1-t^{-1})^2$};
\node (24) at (2,-2) {$1-t^{-1}$};
\node (31) at (-3,-3) {$1-t^{-1}$};
\node (32) at (-1.75,-3) {$(1-t^{-1})^2$};
\node (33) at (-1.35,-2.75) {$1-t^{-1}$};
\node (34) at (-.8,-3) {$(1-t^{-1})^2$};
\node (35) at (.8,-3) {$(1-t^{-1})^2$};
\node (36) at (1.75,-3) {$(1-t^{-1})^2$};
\node (37) at (3,-3) {$1-t^{-1}$};
\path[->,font=\scriptsize,inner sep=1]
 (0) edge node[midway,fill=white]{$1$} (11)
 (0) edge node[midway,fill=white]{$2$} (12)
 (11) edge node[midway,fill=white]{$1$} (21)
 (11) edge node[midway,fill=white]{$2$} (22)
 (12) edge node[midway,fill=white]{$1$} (23)
 (12) edge node[midway,fill=white]{$2$} (24)
 (21) edge node[midway,fill=white]{$1$} (31)
 (21) edge node[midway,fill=white]{$2$} (32)
 (22) edge node[midway,fill=white]{$1$} (33.90)
 (22) edge node[midway,fill=white]{$2$} (35.90)
 (23) edge node[midway,fill=white]{$1$} (34.90)
 (23) edge node[midway,fill=white]{$2$} (36.95)
 (24) edge node[midway,fill=white]{$1$} (36.85)
 (24) edge node[midway,fill=white]{$2$} (37);
\end{tikzpicture}
\caption{The coefficients for the top part of $\mathcal{T}(\infty)$ in type $G_2$.  Compare with Figure \ref{fig:G2}.}
\label{fig:G2t}
\end{figure}

\section{Applications}\label{sec:app}

Throughout this section, we let $q$ be a formal indeterminate.  In \cite{Luszt:83}, Lusztig defined a $q$-analogue of Kostant's partition function as follows.  For $\mu\in P$, set
\[
\PP(\mu;q) := \sum_{\substack{(c_1,\dots,c_N)\in\ZZ_{\ge0}^N \\ \mu=c_1\beta_1+\cdots+c_N\beta_N}} q^{c_1+\cdots+c_N},
\]
where $\{\beta_1,\dots,\beta_N\} = \Phi^+$.  Note that $\PP(\mu;q) =0$ if $\mu\notin Q^+$.  We immediately obtain a way to write this $q$-analogue as a sum over $\TT(\infty)$.  

\begin{definition}
For $T \in \TT(\infty)$, define $|T|$ to be the number of boxes in $T^\sharp$, counting  a box $\boxed{\overline\imath}$ in the $i$th row, if any,  with multiplicity $2$ for each $i$.
\end{definition}

\begin{proposition} \label{prop:1117}
For $\mu\in Q^+$, we have
\[
\PP(\mu;q) = \sum_{\substack{T\in\TT(\infty)\\ -\wt(T)=\mu}} q^{|T|}.
\]
\end{proposition}

\begin{proof}
We write $\Xi(T) = c_1(\beta_1) + \cdots + c_N(\beta_N)$, where $\{\beta_1,\dots,\beta_N\} = \Phi^+$, so that $\wt(T) = -c_1\beta_1 - \cdots - c_N\beta_N$.  Then, from the definition of $\Xi$, one can see that  $|T| = c_1+\dots+c_N$. Now the assertion is clear from the definition of $\PP(\mu;q)$.
\end{proof}

The above proposition enables us to write the Kostka-Foulkes polynomial $K_{\lambda,\mu}(q)$ ($\lambda,\mu\in P^+$)  in terms of $\TT(\infty)$. Namely, we have
\begin{align*}
K_{\lambda,\mu}(q) 
&:= \sum_{w\in W}(-1)^{\ell(w)}\PP(w(\lambda+\rho)-(\mu+\rho);q) \\
&= \sum_{w\in W}(-1)^{\ell(w)} \sum_{\substack{T\in\TT(\infty) \\ -\wt(T) = w(\lambda+\rho)-(\mu+\rho)}} q^{|T|}.
\end{align*}

For $\mu \in P$, let $W_\mu = \{ w \in W : w\mu = \mu \}$, and set $W_\mu(q) = \sum_{w\in W_\mu} q^{\ell(w)}$.
Applying Theorem \ref{main}, we can write the Hall-Littlewood function $P_\mu(\zz; q)$ as a sum over $\TT(\infty)$. That is, we have
\begin{align*}
P_{\mu}(\zz; q) 
&:= \frac{1}{W_\mu(q)}\sum_{w\in W} w \left ( \zz^{\mu} \prod_{\alpha\in\Phi^+} \frac{1-q\zz^{-\alpha}}
{1-\zz^{-\alpha} } \right ) \\
&=  \frac{1}{W_\mu(q)} \sum_{T\in\TT(\infty)} (1-q)^{\seg(T)} \sum_{w\in W} \zz^{w(\mu+\wt(T))}.
\end{align*}
Recall that $K_{\lambda, \mu}(q)$ provides a transition matrix from $P_\mu(\zz; q)$ to the character $\chi_\lambda(\zz)$. More precisely, we have
\[ \chi_\lambda(\zz) = \sum_{\substack{\mu\in P^+ \\ \mu\le\lambda}} K_{\lambda,\mu}(q) P_\mu(\zz;q) .\]
We hope to establish other applications of our formulas to the study of symmetric functions and related topics in future work.

\section{Sage implementation}
The listed second author, together with Travis Scrimshaw, have implemented $\TT(\infty)$ into Sage \cite{combinat,sage}, as well as the statistics $\seg(T)$ and $|T|$.  We conclude here with some examples using the newly developed code.

\begin{ex}
We recreate the data from Example \ref{ex:dseg}.
\begin{verbatim}
sage: Tinf = InfinityCrystalOfTableaux("D4")
sage: row1 = [1,1,1,1,1,1,1,1,1,2,2,-3,-3,-1,-1,-1]
sage: row2 = [2,2,2,2,3,-4,-3,-3]
sage: row3 = [3,-4,-3]
sage: T = Tinf(rows=[row1,row2,row3])
sage: T.weight()
-10*alpha[1] - 12*alpha[2] - 8*alpha[3] - 10*alpha[4]
sage: [T.epsilon(i) for i in T.index_set()]
[5, 0, 4, 6]
sage: [T.phi(i) for i in T.index_set()]
[-3, 4, 0, -2]
sage: T.e(1).pp()
  1  1  1  1  1  1  1  1  1  2 -3 -3 -1 -1 -1
  2  2  2  2  3 -4 -3 -3
  3 -4 -3
sage: T.f(4).pp()
  1  1  1  1  1  1  1  1  1  1  2  2 -3 -3 -1 -1 -1
  2  2  2  2  2  3 -4 -3 -3
  3 -4 -4 -3
sage: T.pp()
  1  1  1  1  1  1  1  1  1  2  2 -3 -1 -1 -1
  2  2  2  2  3 -4 -3 -3
  3 -4 -3
sage: T.reduced_form().pp()
  2  2 -3 -1 -1 -1
  3 -4 -3 -3
 -4 -3
sage: T.seg()
9
sage: T.content()
16
\end{verbatim}
The crystal graph $\TT(\infty)$ down to depth $3$ is outputted using the following commands.
\begin{verbatim}
sage: Tinf = InfinityCrystalOfTableaux("D4")
sage: S = Tinf.subcrystal(max_depth=3)
sage: G = Tinf.digraph(subset=S)
sage: view(G,tightpage=True)
\end{verbatim}
\end{ex}

\bibliography{GKformulaABCDG}{}
\bibliographystyle{amsplain}
\end{document}